\documentclass {article} 
\usepackage {amsthm,amsmath,amsfonts,appendix,ytableau}

	\newtheorem{thm}{Proposition}[section]
	\newtheorem{lemma}[thm]{Lemma}
	\newtheorem{corollary}{Corollary}
	\theoremstyle{definition}
	\newtheorem{defn}[thm]{Definition}
	\theoremstyle{remark}
	\newtheorem{example}[thm]{Example}
	\newtheorem{rem}[thm]{Remark}
	
\begin{document}   
 
	\begin{center}   
	{\bf  Counting Cliques in Finite Distant Graphs \\}   
	{\em Tim\ Silverman\\}   
	\end{center}   

\begin{abstract} 
\noindent  
We state and prove some counting formulas relating to cliques in the distant graphs of projective lines over finite rings. As a preliminary to this, we prove a decomposition theorem for the graphs in terms of the direct-product decomposition of their rings.


\end{abstract}

\section{Introduction}   
   
A projective line over a ring can be made into a graph, the {\bf distant graph}, where the vertices are the points of the projective line, and where an edge exists between two vertices when the points are ``distant" (for which, see below). In \cite{SmallComm} (for the commutative case) and \cite{SmallNonc} (for the noncommutative case) there are tabulated certain vital statistics relating to the counts of cliques in distant graphs over finite rings of small order (up to $32$). However, these papers omit to mention that there exist polynomial formulas for these vital statistics, in terms of the sizes of various objects associated with the rings. In this paper, we derive and discuss these formulas.

In Section 1, we recapitulate existing material on the distant graphs of projective lines over general rings and derive certain general results that we will use in subsequent sections. From the point of view of later sections, the most useful result in this section is our decomposition theorem, Proposition \ref{decomposition}, which asserts that the distant graph of a direct product of rings is the tensor product of the distant graphs of the direct factors.

In Section 2, we derive formulas for the commutative case, and in Section 3 for the general case. In Section 4 we exhibit a relationship between some coefficients of our counting polynomials and some other, better known, coefficients related to partitions. Some incidental extras appear in the appendices.

We introduce graphs first. A {\bf simple} graph is an undirected graph with no multiple edges or self loops (we do not assume here that the vertex set is finite). Each graph we shall be dealing with will be {\it either} a simple graph, {\it or} the graph with one vertex and a single loop, which we shall call $T$. (The presence of this odd exception will, we hope, appear less {\it ad hoc} later on.) Given two vertices $a_1$ and $a_2$, let us write $a_1\sim a_2$ if there is an edge between them (or $a_1 \underset{A}\sim a_2$ if we need to specify the graph $A$). A map between graphs is a map $f$ of the underlying sets such that $a_1\sim a_2\Rightarrow f\left(a_1\right)\sim f\left(a_2\right)$. Among the various ways to define a product of graphs, the one we are interested in here is the so-called {\bf tensor product}: the vertex set of the tensor product $A\times B$ is the cartesian product of the vertex sets of $A$ and $B$, and $\langle a_1,b_1\rangle\underset{A\times B}\sim\langle a_2,b_2\rangle\Leftrightarrow a_1\underset{A}\sim a_2$ {\it and} $b_1\underset{B}\sim b_2$.

\begin{rem}
Since there are no loops (except in $T$), no map of graphs can send two adjacent vertices of a graph to the same vertex of its image (unless the image is $T$). The loop also prevents there from being any maps out of $T$ (other than the identity).
\end{rem}

Now we bring in rings. Let $R$ be a ring with identity. Over this, we can construct the ring of $2\times2$ matrices, $\mathrm{M}_2\left(R\right)$, and sitting inside this is $\mathrm{GL}_2\left(R\right)$, the group of invertible $2\times2$ matrices. Let $\mathrm{GL}_2\left(R\right)$ act on $R^2$ from the right. Sitting inside $R^2$ are free modules of rank $1$, and, among these, the modules which also have a free rank $1$ complement, such as $R\left(1,0\right)$, form a single orbit under the action of $\mathrm{GL}_2\left(R\right)$. We say that this orbit is the set of points of the {\bf projective line} over $R$, $\mathbb{P}\left(R\right)$ (see Section 2 of \cite{ProjRep}). Each point in $\mathbb{P}\left(R\right)$ is of the form $R\left(a,b\right)$ for some $\left(a,b\right)$, {\it viz.} the image of $\left(1,0\right)$ under some element of $\mathrm{GL}_2\left(R\right)$ of the form $\left(\begin{array}{ccc}a&b\\c&d\end{array}\right)$. Such a pair $\left(a,b\right)$ is called {\bf admissible}. Every admissible pair is {\it unimodular}, that is $\exists x,\exists y:a x+b y=1$.

\begin{thm} \label{admiss-unit}

Two admissible pairs $\left(a,b\right)$ and $\left(a',b'\right)$ generate the same point iff there is a unit $u$ with $\left(a,b\right)=u\left(a',b'\right)$.

\end{thm}

\begin{proof}

This combines parts 1 and 2 of Proposition 2.1 of \cite{ProjRep}.\end{proof}

Since $\left(\begin{array}{ccc}0&1\\1&0\end{array}\right)\in \mathrm{GL}_2\left(R\right)$, one of the points of $\mathbb{P}\left(R\right)$ is $R\left(0,1\right)$. Consider the images of the pair $\left\langle R\left(1,0\right),R\left(0,1\right)\right\rangle$ under an element of $\mathrm{GL}_2\left(R\right)$. The elements of such pairs are said to be {\bf distant} from each other, and distantness gives a graph structure on the points of $\mathbb{P}\left(R\right)$, on which $\mathrm{GL}_2\left(R\right)$ acts as graph automorphisms. If $p$ and $q$ are mutually distant, we write $\mathbf{p\triangle q}$. Note that admissible pairs $\left(a,b\right)$ and $\left(c,d\right)$ generate mutually distant points just when $\left(\begin{array}{ccc}a&b\\c&d\end{array}\right)\in \mathrm{GL}_2(R)$.

Now, a ring homomorphism $f:R\rightarrow S$ gives rise to a homomorphism of the corresponding matrix rings $\mathrm{M}_2\left(R\right)\rightarrow \mathrm{M}_2\left(S\right)$ via $\left(\begin{array}{ccc}a&b\\c&d\end{array}\right)\rightarrow\left(\begin{array}{ccc}a^f&b^f\\c^f&d^f\end{array}\right)$, and this gives rise to a group homomorphism $\mathrm{GL}_2\left(R\right)\rightarrow \mathrm{GL}_2\left(S\right)$. Likewise, there is an induced module homomorphism from $R^2$ to $S^2$ which sends a point of $\mathbb{P}\left(R\right)$ to some submodule of $S^2$. Let a point of $\mathbb{P}\left(R\right)$ be of the form $R\left(a,b\right)$ for some admissible $\left(a,b\right)$. Then the image of $R\left(a,b\right)$ will lie in the submodule $S\left(a^f,b^f\right)$, and $\left(a^f,b^f\right)$ is itself admissible, being the image of $\left(1,0\right)\in S^2$ under the action of $\left(\begin{array}{ccc}a^f&b^f\\c^f&d^f\end{array}\right)\in \mathrm{GL}_2\left(S\right)$, so we get a map from the points of $\mathbb{P}\left(R\right)$ to the points of $\mathbb{P}\left(S\right)$. Moreover (Proposition 3.1 of \cite{ProjRep}), this map preserves the graph structure. For if $R\left(a,b\right)\triangle R\left(c,d\right)$ then there is an element $\left(\begin{array}{ccc}a&b\\c&d\end{array}\right)\in \mathrm{GL}_2\left(R\right)$. Then there is an element $\left(\begin{array}{ccc}a^f&b^f\\c^f&d^f\end{array}\right)\in \mathrm{GL}_2\left(S\right)$, so $S\left(a^f,b^f\right)\triangle S\left(c^f,d^f\right)$. Finally, the action of $\mathrm{GL}_2\left(R\right)$ on $\mathbb{P}\left(R\right)$ is carried to an action of $\mathrm{GL}_2\left(S\right)$ on $\mathbb{P}\left(S\right)$. The image of $\mathrm{GL}_2(R)$ in the group of automorphisms of the graph is (by definition, but in agreement with the already-defined case where $R$ is a field) $\mathrm{PGL}_2(R)$.

Note that the trivial ring is sent to the one-vertex, one-loop graph $T$. The exceptional property of the trivial ring, $1=0$, is precisely the reason why $\left(1,0\right)=\left(0,1\right)$ in the free rank-$2$ module over the trivial ring, which is in turn the reason why the single vertex of its projective line has the exceptional property of being adjacent to itself. And that is why we want $T$ despite its oddity.

\begin{thm} \label{decomposition}

Let $R_1$ and $R_2$ be rings and $R_1\times R_2$ be their direct product. Then $\mathbb{P}\left(R_1\times R_2\right)=\mathbb{P}\left(R_1\right)\times\mathbb{P}\left(R_2\right)$. That is, (finite) direct products of rings give rise to tensor products of graphs.

\end{thm}

\begin{proof}

Let $R\cong R_1\times R_2$. Then $\mathrm{M}_2\left(R\right)\cong \mathrm{M}_2\left(R_1\right)\times \mathrm{M}_2\left(R_2\right)$ and $\mathrm{GL}_2\left(R\right)\cong \mathrm{GL}_2\left(R_1\right)\times \mathrm{GL}_2\left(R_2\right)$. (This simple but crucial point is made in a more general context in section 2 of \cite{GenLin}). Now, the points of the $\mathbb{P}\left(R\right)$ are of the form $R\left(a,b\right)$ where $\left(a,b\right)$ is admissible. But let $a=\left\langle a_1,a_2\right\rangle$ and $b=\left\langle b_1,b_2\right\rangle$ for $a_1,b_2\in R_1$ and $a_2,b_2\in R_2$. Then $\left(a,b\right)$ is admissible precisely when $\left(a_1,b_1\right)$ and $\left(a_2,b_2\right)$ are admissible. For $\exists c,d\in R$ such that $\left(\begin{array}{ccc}a&b\\c&d\end{array}\right)\in \mathrm{GL}_2\left(R\right)$ precisely when $\exists c_1,d_1\in R_1,\exists c_2,d_2\in R_2$ such that $\left(\begin{array}{ccc}a_1&b_1\\c_1&d_1\end{array}\right)\in \mathrm{GL}_2\left(R_1\right)$ and $\left(\begin{array}{ccc}a_2&b_2\\c_2&d_2\end{array}\right)\in \mathrm{GL}_2\left(R_2\right)$. So $\mathbb{P}\left(R\right)\cong\mathbb{P}\left(R_1\right)\times\mathbb{P}\left(R_2\right)$ as a set. Moreover, the same argument proves that $\left(a,b\right)\triangle\left(c,d\right)$ precisely when $\left(a_1,b_1\right)\triangle\left(c_1,d_1\right)$ and $\left(a_2,b_2\right)\triangle\left(c_2,d_2\right)$, so that $\mathbb{P}\left(R\right)\cong\mathbb{P}\left(R_1\right)\times\mathbb{P}\left(R_2\right)$ as a graph. Moreover $\mathrm{GL}_2\left(R\right)\cong \mathrm{GL}_2\left(R_1\right)\times \mathrm{GL}_2\left(R_2\right)$, with each factor of the group acting on the corresponding factor of the product graph.

\end{proof}

Finite direct products of rings are the same as finite direct sums, so we need not distinguish them in future sections.

\begin{rem}

In the language of category theory, the discussion immediately prior to Proposition \ref{decomposition} implies that taking the projective line over a ring is a functor from the category of unital rings to the category whose objects are simple graphs and $T$, and whose morphisms are graph homomorphisms. (It is not usual to throw $T$ in with simple graphs, but it does no harm.) Moreover, Proposition \ref{decomposition} (relating the respective categorical products), together with the trivial ring being sent to $T$ (both being terminal), implies that this functor preserves all finite products. (Indeed on some subcategories, e.g. finite commutative rings, it can be shown to preserve all finite limits.)

\end{rem}

\section{Counting formulas: the commutative case}

Every finite commutative ring is the direct sum of local rings, so, by Proposition \ref{decomposition}, every distant graph over a finite commutative ring is the tensor product of the distant graphs over local rings; hence our chief task is to characterise the latter. A little algebra shows that (even in the most general case) the condition for $R\left(1,0\right)\triangle R\left(a,b\right)$, i.e. for $\left(\begin{array}{ccc}1&0\\a&b\end{array}\right)$ to be invertible, is just that $b$ should be a unit, and without loss of generality we can assume that $b=1$. Likewise $R\left(a,1\right)\triangle R\left(b,1\right)$ just when $a-b$ is a unit. But this requirement takes particularly simple form in a local ring, where it says that $a$ and $b$ lie in different cosets of the maximal ideal. We can now conveniently describe the distant graph of a local ring in terms of its complement. We use vertical bars to denote cardinalities.

\begin{thm} \label{commloc}

The distant graph of a finite local ring $R$ with Jacobson radical $J$ consists of the complement of the disjoint union of $\displaystyle\frac{\left\vert R\right\vert}{\left\vert J\right\vert}+1$ copies of the complete graph on $\left\vert J\right\vert$ vertices.

\end{thm}

\begin{proof}

Let $R$ be a local ring and let its maximal ideal be $J$. Every point of $\mathbb{P}\left(R\right)$ can be generated by an admissible pair of one of the forms $\left(a,1\right)$ or $\left(1,a\right)$. If a point is generated by an admissible pair both of whose components are units (i.e. it can be represented by both forms), let us represent it with the first form; thus we have one point generated by each pair of the form $\left(r,1\right)$ for $r\in R$, and one point generated by each pair of the form $\left(1,a\right)$, $a\in J$.

Now, two points of the first form, $\left(a,1\right)$ and $\left(b,1\right)$, are distant just if $a$ and $b$ lie in different cosets of $J$, or, to put it another way, each coset of $J$ gives a complete graph of order $\left\vert J\right\vert$ in the complementary graph. Points represented by the second form are all distant to those with the first form, because $\left(\begin{array}{cc}a&1\\1&b\end{array}\right)$ with $b\in J$ has as determinant the sum of $1$ and an element of $J$, and is therefore invertible. But they are non-distant to each other, because $\left(\begin{array}{cc}1&a\\1&b\end{array}\right)$ with $a,b\in J$ has as determinant an element of $J$, so is not invertible. Hence the set of such elements forms one more complete subgraph of order $\left\vert J\right\vert$ in the complementary graph.\end{proof}

Note that this graph structure depends only on the cardinalities $\left\vert R\right\vert$ and $\left\vert J\right\vert$ and not on any other details of the ring structure. Since we are dealing with finite local commutative rings, we can say a little more about the cardinalities. From Theorem 2 of \cite{FinRing}, there must be a prime number $p$ and two positive integers $n$ and $r$ such that $\left\vert R\right\vert=p^{n r}$ and $\left\vert J\right\vert=p^{\left(n-1\right)r}$. Conversely, there is at least one such ring for any choice of $p$, $n$, and $r$, namely the Galois ring (see the remarks near the beginning of section 3 of \cite{FinRing}). Hence all graphs of the appropriate form correspond to some finite commutative ring.

\begin{thm} \label{commcliq}

Let $R$ be a finite commutative ring with identity, and let $R_i$ be its local summands, of order $\left\vert R_i\right\vert$. Let the corresponding Jacobson radicals be $J$ of order $\left\vert J\right\vert$ and $J_i$, of order $\left\vert J_i\right\vert$, and let $q_i=\displaystyle\frac{\left\vert R_i\right\vert}{\left\vert J_i\right\vert}$. Then the following hold.

a) The number of $k$-cliques in $\mathbb{P}(R)$ is $\left\vert J\right\vert^k\prod_i\left(\begin{array}{ccc}q_i+1\\k\end{array}\right)$.

b) The number of $\left(k+1\right)$-cliques containing a given $k$-clique is $\left\vert J\right\vert\prod_i\left(q_i+1-k\right)$.

c) The maximal order of a clique is $\mathrm{min}\left(q_i\right)+1$.

\end{thm}

\begin{proof}
Consider one of the local rings $R_i$. From Proposition \ref{commloc}, the complement of its distant graph consists of $q_i+1$ copies of the complete graph on $\left\vert J_i\right\vert$ vertices. A clique in the distant graph then contains at most one point from any of the complete graphs. Thus there are ${\left\vert J_i\right\vert}^k\left(\begin{array}{ccc}q_i+1\\k\end{array}\right)$ $k$-cliques altogether; and the number of $\left(k+1\right)$-cliques containing a given $k$-clique is $\left\vert J_i\right\vert\left(q_i+1-k\right)$.

$k$ points of a product graph are mutually distant precisely when their components in each factor are mutually distant, so the number of $k$-cliques in the projective line over a finite ring is just the product of the number in each local summand ring. Since $J$ is the cartesian product of the $J_i$, $\left\vert J\right\vert=\displaystyle\prod_i\left\vert J_i\right\vert$, the cardinality $\left\vert J\right\vert$ can be pulled out of the product to give the formulas stated above.

When $k=q_i+1$, obviously $\left\vert J_i\right\vert\left(q_i+1-k\right)=0$, so $q_i+1$ is the maximal order of a clique in $R_i$. Since a clique in a tensor product graph must project to a clique in each of its factors, the maximal order of a clique in the product graph cannot exceed the maximal order in any of the factor graphs, hence is the minimum of the maxima.\end{proof}

We now take a definition from Section 2 of \cite{SmallComm}.

\begin{defn}

The {\bf distant-set} of a point is the set of points distant to it, and its {\bf neighbourhood} is the complement of its distant-set.

\end{defn}

We find listed in \cite{SmallComm}, among the properties of projective spaces of small rings, the cardinalities of the intersections of the neighbourhoods of $n$ mutually distant points, a number there denoted by $\cap n\mathrm{N}$. There is a general formula for this quantity (with the same notation as in Proposition \ref{commcliq}).

\begin{thm}

$\cap n\mathrm{N}=\left\vert J\right\vert\displaystyle\sum_{k=0}^n\left(-1\right)^k\left(\begin{array}{ccc}n\\k\end{array}\right)\displaystyle\prod_i\left(q_i+1-k\right)$

\end{thm}

\begin{proof} The intersection of a set of complements is the complement of their union, and the cardinality of the union can be calculated by inclusion-exclusion. A set of $k$ mutually distant points is a $k$-clique and so the number of points in the intersection of all their distant-sets is just the number of $k+1$-cliques containing that $k$-clique, {\it viz.} $\left\vert J\right\vert\prod_i{\left(q_i+1-k\right)}$. Given $n$ mutually distant points, there are $\left(\begin{array}{ccc}n\\k\end{array}\right)$ $k$-fold intersections among their distant sets, and each $k$-fold intersection has $\left\vert J\right\vert\prod_i{\left(q_i+1-k\right)}$ points, so by inclusion-exclusion, their union contains

\[-\left\vert J\right\vert\displaystyle\sum_{k=1}^n{\left(-1\right)^k\left(\begin{array}{ccc}n\\k\end{array}\right)\displaystyle\prod_i\left(q_i+1-k\right)}\]

\noindent points. The total number of points in the projective line, viz. $\left\vert J\right\vert\prod_i{\left(q_i+1\right)}$, is just the value of the expression inside the sum for $k=0$, so, taking the complement of the union, we have:

\[\cap n\mathrm{N}=\left\vert J\right\vert\displaystyle\sum_{k=0}^n\left(-1\right)^k\left(\begin{array}{ccc}n\\k\end{array}\right)\displaystyle\prod_i\left(q_i+1-k\right)\]\end{proof}

There is not much to be said about this sum, except for the quirky combinatorial fact, which we prove in Appendix \ref{lacun}, that $p\vert\cap n\mathrm{N}$ for every prime $p\le n$. Hence $\cap5\mathrm{N}$ is always a multiple of $30$, and so forth.

\section{Counting formulas for general finite rings}

We now turn to the non-commutative case. Some results for small non-commutative rings are briefly tabulated in \cite{SmallNonc}.

Let $R$ be a finite ring and let $J$ be its Jacobson radical. We will make use of a definition and some theorems from \cite{RadPara}.

\begin{defn}

Two points $p$ and $q$ of a projective line are {\bf radically parallel}, $p\|q$, just if they have the same distant sets.

\end{defn}

This is not the same as the definition in \cite{RadPara}, but by their Corollary 2.3 it is equivalent.

\begin{thm}

The points of the form $R\left(1,a\right)$ for $a\in J$ form an equivalence class under the relation of radical parallelism.

\end{thm}

\begin{proof}

This follows immediately from \cite{RadPara} Theorem 2.1.\end{proof}

\begin{thm}

Two points of the form $R\left(1,a\right)$ for $a\in J$ are not mutually distant.

\end{thm}

\begin{proof}

We can not possibly have $\left(\begin{array}{ccc}1&a\\1&b\end{array}\right)\in\mathrm{GL}_2(R)$ with $a,b\in J$, since it sends $\left(r,s\right)$ to $\left(r+s,r a+s b\right)$, and the second element of the latter pair will always lie in the Jacobson radical.\end{proof}

Since the distant graph is vertex transitive, every vertex lies in such a set of order $\left\vert J\right\vert$ of mutually non-distant points sharing their distant-sets. Since the points of such a set are mutually non-distant, taking the quotient of the graph by the identification of radically parallel points gives a perfectly good surjective morphism of graphs. We might hope that this quotient is the distant graph of the projective line of the quotient of $R$ by $J$, and Theorem 2.2 of \cite{RadPara}, with the immediately preceding discussion, shows just this.

\begin{thm}

$\mathbb{P}\left(R/J\right)\cong\mathbb{P}\left(R\right)/\|$.

\end{thm}

This extremely convenient result means that, in the finite (or, generally, Artinian) case, we can easily deal with projective lines over rings with non-trivial Jacobson radical by reducing to the quotient, and only need to worry about semisimple rings, which, by the Artin-Wedderburn theorem, means direct sums of matrix rings over finite fields (in the finite case). As we already know how to deal with direct sums, our remaining task is to characterise, so far as possible, the projective lines over matrix rings over finite fields.

Let us start by counting the points of $\mathbb{P}\left(\mathrm{M}_m\left(q\right)\right)$, the projective line over the ring of $m\times m$ matrices over the field $\mathbb{F}_q$ of order $q$.

\begin{thm}

$\left\vert\mathbb{P}\left(\mathrm{M}_m\left(q\right)\right)\right\vert$ is the $q$-binomial coefficient $\left\lbrack\begin{array}{ccc}2m\\m\end{array}\right\rbrack_q=\displaystyle\prod_{k=0}^{m-1}\displaystyle\frac{q^{2m-k}-1}{q^{k+1}-1}$.

\end{thm}

\begin{proof}

We shall prove this by assigning an admissible pair to each point. First we observe, from inspecting the form of the multiplication, that $\mathrm{M}_2\left(\mathrm{M}_m(q)\right)$ is a lightly disguised version of $\mathrm{M}_{2m}(q)$, and likewise $\mathrm{GL}_2\left(\mathrm{M}_m(q)\right)$ is $\mathrm{GL}_{2m}(q)$. Each admissible pair forms the top row of an element of $\mathrm{GL}_2\left(\mathrm{M}_m(q)\right)$, but we can treat these top rows equivalently as the top $m$ rows of an element of $\mathrm{M}_{2m}(q)$, i.e. as an $m\times2m$ matrix in its own right. That these $m$ rows actually belong to an invertible $2m\times2m$ matrix amounts to the requirement that they span an $m$-dimensional space (that is, that there are no linear relations among them).

Now, two admissible pairs are equivalent just if they are related by multiplication on the left by an invertible element of the ring, i.e. by an element of $\mathrm{GL}_m(q)$. The left action of $\mathrm{GL}_m(q)$ interchanges every basis of the given subspace while preserving the subspace itself. Hence we have one point of $\mathbb{P}\left(\mathrm{M}_m(q)\right)$ for each $m$-dimensional subspace of a $2m$ dimensional space over $\mathbb{F}_q$. The number of these subspaces is the $q$-binomial coefficient.\end{proof}

It is also true, by continuing the argument of the proof to all $2m$ rows of the matrix, that two points are distant just when the corresponding subspaces have trivial intersection. (These facts about matrix rings are mentioned in passing in \cite{Spreads}).

Here are the values for the smallest $m$, after which the polynomials become more complicated:

\begin{center}
$\begin{array}{|c|c|c|}
\hline
\mathbf{m}&\mathbf{\left\lbrack\begin{array}{ccc}\mathbf{2m}\\\mathbf{m}\end{array}\right\rbrack_q}\\
\hline
0&1\\
\hline
1&q+1\\
\hline
2&q^4+q^3+2q^2+q+1\\
\hline
3&q^9+q^8+2q^7+3q^6+3q^5+3q^4+3q^3+2q^2+q+1\\
\hline
\end{array}$
\end{center}

(Of course, $m=0$ corresponds to the trivial ring, and $m=1$ to the field of order $q$.)

\begin{thm}

For a matrix ring $\mathrm{M}_m(q)$,

a) The number of points distant to a given point is $q^{m^2}$.

b) The number of points distant to a pair of mutually distant points is $\displaystyle\prod_{k=0}^{m-1}{\left(q^m-q^k\right)}$.

c) $\cap1N=\displaystyle\prod_{k=0}^{m-1}\displaystyle\frac{q^{2m-k}-1}{q^{k+1}-1}-q^{m^2}$

d) $\cap2N=\displaystyle\prod_{k=0}^{m-1}\displaystyle\frac{q^{2m-k}-1}{q^{k+1}-1}-2q^{m^2}+\displaystyle\prod_{k=0}^{m-1}{\left(q^m-q^k\right)}$.

For a product of matrix rings $\displaystyle\prod_i{M_{m_i}\left(q_i\right)}$

e) $\cap1N=\displaystyle\prod_i{\displaystyle\prod_{k=0}^{m_i-1}\displaystyle\frac{q_i^{2m_i-k}-1}{q_i^{k_i+1}-1}}-\displaystyle\prod_i{q_i^{m_i^2}}$

f) $\cap2N=\displaystyle\prod_i{\displaystyle\prod_{k=0}^{m_i-1}\displaystyle\frac{q_i^{2m_i-k}-1}{q_i^{k_i+1}-1}}-2\displaystyle\prod_i{q_i^{m_i^2}}+\displaystyle\prod_i\displaystyle\prod_{k=0}^{m_i-1}{\left(q_i^{m_i}-q_i^k\right)}$

\end{thm}

\begin{proof}

(Note that a) and b) follow straightforwardly from results for a general ring $R$.)

a) As the distant graph is vertex transitive, each vertex has the same degree. Hence we can choose, for example, to count the points distant to $R\left(1,0\right)$. The points distant to this are just $R\left(r,1\right)$ for $r\in R$, so their number is just the order of the ring; for $\mathrm{M}_m(q)$, this is $q^{m^2}$.

b) As the distant graph is also edge transitive, the number of points distant to two mutually distant points is always the same, so we can choose points $R\left(1,0\right)$ and $R\left(0,1\right)$. The points distant to both of these are just those of the form $R\left(u,1\right)$, where $u$ is a unit of $R$. The units of $\mathrm{M}_m(q)$ are just the elements of $\mathrm{GL}_m(q)$, and $\left\vert \mathrm{GL}_m(q)\right\vert=\displaystyle\prod_{k=0}^{m-1}{\left(q^m-q^k\right)}$.

c) and d) These follow from inclusion-exclusion on the clique counts, as with the commutative case.

e) and f) Each term is a count of $k$-cliques, for $k$ successively equal to $1$, $2$ and (in the second case) $3$, which is given by the products of the clique-counts for the summand rings.\end{proof}

For small $m$, this gives us the following table for $\mathrm{M}_m(q)$:

\begin{center}
$\begin{array}{|c|c|c|}
\hline
\mathbf{m}&\mathbf{\cap1N}&\mathbf{\cap2N}\\
\hline
0&0&0\\
\hline
1&1&0\\
\hline
2&q^3+2q^2+q+1&q^2+2q+1\\
\hline
3&q^8+2q^7+3q^6+3q^5+3q^4+3q^3+2q^2+q+1&q^7+3q^6+4q^5+4q^4+2q^3+2q^2+q+1\\
\hline
\end{array}$
\end{center}

\begin{rem} Because the distant-graph over the trivial ring has a self-loop, the formulas for $m=0$ imply that there is $1$ $k$-clique for every $k$. In the other cases, with no self-loops, clique-counting works properly.\end{rem}

We can also try to count $4$-cliques and hence calculate $\cap3N$. As $\mathrm{GL}_2(R)$ acts transitively on triangles for any projective line, we can, without loss of generality, take three of the points of the clique to be $\left(1,0\right)$, $\left(0,1\right)$ and $\left(1,1\right)$ and ask how many ways there are to extend this. Points distant to all $3$ of these points must be of the form $\left(u,1\right)$ where the $u$ are invertible elements such that $u-1$ is also invertible; in terms of matrices, this means matrices with no eigenvalues equal to $0$ or $1$. We can handle this by counting matrices which {\it do} have such eigenvalues, and excluding them.

For this purpose, we will introduce a lemma on inclusion-exclusion. Consider the following setup: a vector space $X$ of dimension $m$ over $\mathbb{F}_q$, a set $\mathcal{G}$, and a relation, {\it capture}, between elements of $\mathcal{G}$ and subspaces of $X$, subject to the following conditions:

a) If $g\in\mathcal{G}$ captures a subspace $U$, then it also captures all subspaces of $U$.

b) The cardinality of the set of elements capturing a given subspace depends only on the dimension of that subspace.

For concreteness, the example we have in mind has $\mathcal{G}$ being the set of endomorphisms of $X$ and ``captures $U$" being ``acts as the identity when restricted to $U$".

Let us write $W_U$ for the set of elements that capture $U$, and $W_{m,k}$ for the cardinality of the set of elements capturing a subspace of dimension $k$ in a space of dimension $m$. Now condition a) implies that if $U\subseteq V$ then $W_U\supseteq W_V$. Let us write $W'_U$ for the subset of $W_U$ whose elements do not lie in $W_V$ for any $V$ properly containing $U$, and $W'_{m,k}$ for the cardinality of $W'_U$ for $U$ of dimension $k$. Then we have the following.

\begin{lemma}\label{incexc}

\[W'_{m,k}=\sum_{i=0}^{m-k}\left(-1\right)^i\left\lbrack\begin{array}{cc}m-k\\i\end{array}\right\rbrack_q q^{\frac{i(i-1)}{2}}W_{m,k+i}\]

\end{lemma}

\begin{proof}

In the $q$-binomial theorem,

\[\prod_{i=0}^{m-1}\left(1+q^i t\right)=\sum_{i=0}^m t^i q^\frac{i\left(i-1\right)}{2}\left\lbrack\begin{array}{c}m\\i\end{array}\right\rbrack_q\]

substitute $t=-1$ to get

\[\sum_{i=0}^{m} \left(-1\right)^i q^\frac{i\left(i-1\right)}{2}\left\lbrack\begin{array}{c}m\\i\end{array}\right\rbrack_q=0\]

Hence

\[\sum_{i=0}^{m-1} \left(-1\right)^i q^\frac{i\left(i-1\right)}{2}\left\lbrack\begin{array}{c}m\\i\end{array}\right\rbrack_q=-\left(-1\right)^m q^\frac{m\left(m-1\right)}{2}\]

Given a subspace $U_k$ of dimension $k$ in a space of dimension $m$, let us calculate $W'_{m,k}$ by inclusion-exclusion. By Poincar\'{e} duality $U_k$ is contained in $\left\lbrack\begin{array}{c}m-k\\m-(k+i)\end{array}\right\rbrack_q=\left\lbrack\begin{array}{c}m-k\\i\end{array}\right\rbrack_q$ subspaces of dimension $k+i$. Also, each of these is captured by $W_{m,k+i}$ elements of $\mathcal{G}$. Hence if each such subspace appears with a factor of $\left(-1\right)^i q^\frac{i(i-1)}{2}$ in the inclusion-exclusion, then the lemma will be proved. We show this by induction.

We include $U_k$ itself once. As $\left(-1\right)^i q^\frac{i(i-1)}{2}=1$ when $i=0$, the induction starts correctly.

Now suppose that each $k+j$-space containing $U_k$ has been included/excluded $\left(-1\right)^j q^\frac{j(j-1)}{2}$ times for $0\le j<i$. Consider an $i$-space $U_i\supset U_k$. For $0\le j<i$, consider the $j$-spaces $U_j$ with $U_k\subseteq U_j\subset U_i$. There are $\left\lbrack\begin{array}{c}i\\j\end{array}\right\rbrack_q$ such subspaces, and each one will have been included $\left(-1\right)^j q^\frac{j(j-1)}{2}$ times. Hence $U_i$ has already been included/excluded $\sum_{j=0}^{i-1} \left(-1\right)^j q^\frac{j\left(j-1\right)}{2}\left\lbrack\begin{array}{c}i\\j\end{array}\right\rbrack_q=-\left(-1\right)^i q^\frac{i\left(i-1\right)}{2}$ times. So to cancel this, we need to add it back in $\left(-1\right)^i q^\frac{i\left(i-1\right)}{2}$ times. This completes the induction.

\end{proof}

\begin{example} Let $\mathcal{G}$ be the set of all endomorphisms, and let an endomorphism capture a subspace if its restriction to that subspace is the zero endomorphism. In an appropriate basis, an endomorphism capturing a subspace of dimension $i$ is given by a matrix with the first $i$ columns all zero (and no restriction on the other columns), so $W_{m,i}=q^{m(m-i)}$. Then the number of invertible endomorphisms is just $W'_{m,0}$. We have, from Lemma \ref{incexc},

$\begin{array}{lll}W'_{m,0}&=&\sum_{i=0}^m\left(-1\right)^i\left\lbrack\begin{array}{c}m\\i\end{array}\right\rbrack_q q^\frac{i(i-1)}{2}q^{m(m-i)}\\&=&q^{m^2}\sum_{i=0}^m\left(-q^{-m}\right)^i\left\lbrack\begin{array}{c}m\\i\end{array}\right\rbrack_q q^\frac{i(i-1)}{2}\\&=&q^{m^2}\prod_{i=0}^{m-1}\left(1-q^{-m}q^i\right)\\&=&\prod_{i=0}^{m-1}\left(q^m-q^i\right)\end{array}$

\noindent where the third step follows from the $q$-binomial theorem. \end{example}

\begin{thm}

In the distant graph of $\mathbb{P}\left(M_m(q)\right)$, the number of $4$-cliques containing a given $3$-clique is

\[\left(-1\right)^mq^\frac{m(m-1)}{2}\displaystyle\sum_{i=0}^m\displaystyle\prod_{j=0}^{m-i-1}\left(1-q^{m-j}\right)\]

or, equivalently,

\[\left(-1\right)^m q^\frac{m(m-1)}{2}\left(\left(1-q^m\right)\left(\left(1-q^{m-1}\right)\ldots\left(\left(1-q^2\right)\left(\left(1-q\right)+1\right)+1\right)\ldots+1\right)+1\right)\]

\end{thm}

\begin{proof}

From the discussion earlier, the number of $4$-cliques containing a given $3$-clique is equal to the number of matrices in $M_m(q)$ which have neither $0$ nor $1$ as an eigenvalue. Let $\mathcal{G}$ be the set of invertible endomorphism of $\mathbb{F}_q^m$, and let an endomorphism capture a subspace if its restriction to that subspace is the identity. Then the number we seek is $W'_{m,0}$.

Now an element which captures the $k$-subspace of vectors whose coordinates (in some suitable basis) are all $0$ after the $k$th will have a matrix of the form

\[\left(\begin{array}{c|c}I&*\\\hline0&G\end{array}\right)\]

\noindent where $I$ is a $k\times k$ identity matrix, $*$ is anything and $G$ is invertible. Hence $W_{m,k}=q^{k(m-k)}\prod_{i=0}^{m-k-1}\left(q^{m-k}-q^i\right)$

Then

\[\begin{array}{lll}W'_{m,0}&=&\displaystyle\sum_{i=0}^m\left(-1\right)^i\left\lbrack\begin{array}{c}m\\i\end{array}\right\rbrack_q q^\frac{i(i-1)}{2}q^{i(m-i)}\displaystyle\prod_{j=0}^{m-i-1}\left(q^{m-i}-q^j\right)\\
&=&\displaystyle\sum_{i=0}^m\left(-1\right)^i\left(\displaystyle\prod_{j=0}^{m-i-1}\frac{q^m-q^j}{q^{m-i}-q^j}\right)q^\frac{i(i-1)}{2}q^{i(m-i)}\displaystyle\prod_{j=0}^{m-i-1}\left(q^{m-i}-q^j\right)\\&=&\displaystyle\sum_{i=0}^m\left(-1\right)^i\left(\displaystyle\prod_{j=0}^{m-i-1}\left(q^m-q^j\right)\right)q^\frac{i(i-1)}{2}q^{i(m-i)}\\&=&\displaystyle\sum_{i=0}^m\left(-1\right)^i\left(\displaystyle\prod_{j=0}^{m-i-1}\left(q^{m-j}-1\right)\right)q^\frac{(m-i)(m-i-1)}{2}q^\frac{i(i-1)}{2}q^{i(m-i)}\\&=&q^\frac{m(m-1)}{2}\displaystyle\sum_{i=0}^m\left(-1\right)^i\left(\displaystyle\prod_{j=0}^{m-i-1}\left(q^{m-j}-1\right)\right)\\&=&\left(-1\right)^mq^\frac{m(m-1)}{2}\displaystyle\sum_{i=0}^m\displaystyle\prod_{j=0}^{m-i-1}\left(1-q^{m-j}\right)
\end{array}\]

The alternative, nested, form is obtained by gathering terms.

\end{proof}

For $k$-cliques with $k>3$, we cannot produce simple counts in the general case, as we can for commutative rings, because these cliques are no longer all equivalent. For instance, when extending a $k$-clique to a $k+1$-clique, the $k$-cliques fall into distinct classes distinguished by the number of ways it is possible to extend them. In particular, there are cliques which are inextensible although they do not have the maximal order (for which, see below). We give simple concrete examples of this sort of thing in Appendices \ref{extensions} and \ref{inextensible}.

For a ring $R$ with non-trivial Jacobson radical $J$, we can calculate $\cap k N$ for $R/J$, then multiply by $\vert J\vert$ to get $\cap k N$ for $R$, since each vertex of $\mathbb{P}(R)$ is just $\vert J\vert$ ``copies" of a vertex of $\mathbb{P}\left(R/J\right)$.

We can also calculate the maximal order of a clique.

\begin{thm}

The maximal order of a clique in $\mathbb{P}\left(\mathrm{M}_m(q)\right)$ is $q^m+1$.

\end{thm}

\begin{proof}

Suppose that the clique contains $\left(1,0\right)$ and $\left(0,1\right)$. (Since the graph is vertex- and edge-transitive, this is no loss of generality.) Then the remaining points of the clique must be of the form $\left(u,1\right)$ for invertible $u\in R$, i.e., in the present context, elements of $\mathrm{GL}_m(q)$. Two elements of this form can belong to the same clique just if $\left(\begin{array}{ccc}u_1&1\\u_2&1\end{array}\right)$ is invertible, which occurs just when $u_1-u_2$ is invertible. So let us temporarily confine ourselves to the subgraph whose vertices are members of $\mathrm{GL}_m(q)$ and whose edges lie between matrices whose difference is invertible, and work with cliques in this graph. Note that for any set of matrices $\left\{u_i\right\}$ forming a clique, and for any $v\in\mathrm{GL}_m(q)$, the set $\left\{v u_i\right\}$ also forms a clique, so all cliques are translates of ones containing the identity matrix.

Now consider, say, the top rows of all the $u_i$ in a clique. These must all be different from one another, lest $u_i-u_j$ not be invertible for some $i,j$. That leaves a maximum of $q^m$ possibilities, but we must exclude the case of all zeroes. This gives an upper bound of $q^m-1$ to the size of the maximal clique. To see that this bound is attained, consider any matrix $u$ whose characteristic polynomial is irreducible over $\mathbb{F}_q$. Then its eigenvalues are all distinct, and each of them is a generator of the group of units $\mathbb{F}_{q^m}^*$, of order $q^m-1$. So we have, for every eigenvalue $\lambda$, $\lambda^i\ne1$ for $0< i< q^m-1$. Now let $u_i=u^i$ for $0\le i< q^m-1$. Suppose $i>j$. Then $u^i-u^j=u^j\left(u^{i-j}-1\right)$, and, by the foregoing, $u^{i-j}$ can have have no eigenvalues equal to $1$ for $i\ne j$, and hence this difference is invertible. So $\left\{u_i\right\}$ is a clique of order $q^m-1$. (It is, however, not hard to exhibit maximal cliques containing $1$ that are not of this form.)

So much for cliques inside $\mathrm{GL}_m(q)$. In $\mathbb{P}\left(\mathrm{M}_m(q)\right)$, we include all points of the form $\left(u_i,1\right)$ together with $\left(1,0\right)$ and $\left(0,1\right)$. Hence the order of a maximal clique is $q^m+1$.\end{proof}

By the same argument as in the commutative case, the maximal order of a clique in a general finite ring $R$ with Jacobson radical $J$ is the minimum value of $q_i^{m_i}+1$ across the matrix-ring summands of $R/J$.

\begin{rem}The number of elements in an $m$-dimensional subspace, excluding the $0$ vector, is $q^m-1$. As the subspaces corresponding to distant points intersect only in the $0$ vector, a $\left(q^m+1\right)$-clique contains $q^{2m}-1$ elements, so, putting back the $0$ vector, we get the whole space. Hence the maximal clique corresponds to a spread, and the above theorem is equivalent to a (rather specialised) case of general theorems on the existence of spreads (e.g. \cite{GenSpreads}, Lemma 2).\end{rem}

\section{Coefficients of clique-extension counts}

In this section we shall prove a curious theorem on the coefficients of the higher-order terms for $k$-clique-extension counts up to $k=3$. We first introduce some definitions and lemmas (there does not appear to be standard snappy terminology for the concepts defined below).

\begin{defn}

Given some set $P_n$ of partitions of a number $n$, we say that the {\bf parity-count} of $P_n$, denoted $\mathrm{PC}\left(P_n\right)$, is the number of partitions in $P_n$ with an even number of parts minus the number of partitions in $P_n$ with an odd number of parts.

\end{defn}

\begin{defn} A {\bf distinct partition} is a partition into numbers no two of which are equal. A $\mathbf{2}${\bf-distinct partition} is a partition whose elements are split across two (possibly empty) subsets, such that each subset consists of distinct elements. An $\mathbf{\left(h,k,\star\right)}${\bf-distinct partition} is a $2$-distinct partition of $h$ such that the first subset contains exactly $k$ elements, while the second subset can contain any number of elements. We denote the set of $\left(h,k,\star\right)$-distinct partitions by $\mathcal{D}_2(h,k,\star)$. \end{defn}

To be clear, with $\left(h,k,\star\right)$-distinct partitions we have partitions whose Young diagrams have $h$ boxes divided into (say) red rows and white rows, such that no two red rows have equal length, no two white rows have equal length, and there are exactly $k$ red rows. For example, here are three $\mathcal{D}_2(h,k,\star)$, with $(h,k)=(4,1)$, $(4,2)$ and $(6, 2)$.

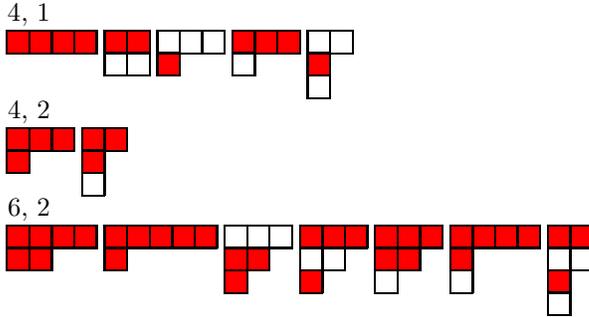
\ytableausetup{smalltableaux}
$4$, $1$

\begin{ytableau}
*(red)\;&*(red)\;&*(red)\;&*(red)\;
\end{ytableau}
\begin{ytableau}
*(red)\;&*(red)\;\\\;&\;
\end{ytableau}
\begin{ytableau}
\;&\;&\;\\*(red)\;
\end{ytableau}
\begin{ytableau}
*(red)\;&*(red)\;&*(red)\;\\\;
\end{ytableau}
\begin{ytableau}
\;&\;\\*(red)\;\\\;
\end{ytableau}

$4$, $2$

\begin{ytableau}
*(red)\;&*(red)\;&*(red)\;\\*(red)\;
\end{ytableau}
\begin{ytableau}
*(red)\;&*(red)\;\\*(red)\;\\\;
\end{ytableau}

$6$, $2$

\begin{ytableau}
*(red)\;&*(red)\;&*(red)\;&*(red)\;\\*(red)\;&*(red)\;
\end{ytableau}
\begin{ytableau}
*(red)\;&*(red)\;&*(red)\;&*(red)\;&*(red)\;\\*(red)\;
\end{ytableau}
\begin{ytableau}
\;&\;&\;\\*(red)\;&*(red)\;\\*(red)\;
\end{ytableau}
\begin{ytableau}
*(red)\;&*(red)\;&*(red)\;\\\;&\;\\*(red)\;
\end{ytableau}
\begin{ytableau}
*(red)\;&*(red)\;&*(red)\;\\*(red)\;&*(red)\;\\\;
\end{ytableau}
\begin{ytableau}
*(red)\;&*(red)\;&*(red)\;&*(red)\;\\*(red)\;\\\;
\end{ytableau}
\begin{ytableau}
*(red)\;&*(red)\;\\\;&\;\\*(red)\;\\\;
\end{ytableau}

We may observe that there are just as many partitions in the third row as in the first two together, and this follows from a general lemma.

\begin{lemma}\label{dist2P}

There is a bijection $f:\mathcal{D}_2(h,k,\star)\rightarrow\mathcal{D}_2(h-k,k,\star)\cup\mathcal{D}_2(h-k,k-1,\star)$ such that if $f(x)\in\mathcal{D}_2(h-k,k,\star)$ then $x$ and $f(x)$ have the same number of rows, while if $f(x)\in\mathcal{D}_2(h-k,k-1,\star)$ then $f(x)$ has one row less than $x$.

\end{lemma}

\begin{proof}

From $x\in\mathcal{D}_2(h,k,\star)$, remove one cell from each of the $k$ red rows, giving a $2$-distinct partition of $h-k$. If there is a (necessarily unique) red row of length $1$ in $x$, then that row disappears and $f(x)\in\mathcal{D}_2(h-k,k-1,\star)$, and has one row less than $x$. If all red rows are longer than $1$, then $f(x)\in\mathcal{D}_2(h-k,k,\star)$ and has the same number of rows as $x$. This easily reversed operation is clearly a bijection.

\end{proof}

From the above, it follows not only that the count of partitions in $\mathcal{D}_2(h,k,\star)$ is the sum of the counts in $\mathcal{D}_2(h-k,k,\star)$ and $\mathcal{D}_2(h-k,k-1,\star)$, but also that the parity count of $\mathcal{D}_2(h,k,\star)$ is the difference of the parity counts: $\mathrm{PC}\left(\mathcal{D}_2(h,k,\star)\right)=\mathrm{PC}\left(\mathcal{D}_2(h-k,k,\star)\right)-\mathrm{PC}\left(\mathcal{D}_2(h-k,k-1,\star)\right)$, since subtracting one row from every partition in a set simply changes the sign of the parity count. (We do not exclude the empty partition from our counts.)

\begin{defn}An $\mathbf{m}${\bf -bounded partition} is a partition whose largest element is no larger than $m$, and which contains no more than $m$ elements. (That is, one whose Young diagram fits into an $m\times m$ grid.)\end{defn}

\begin{lemma}\label{distinctP}

The coefficient of $q^{m^2-h}$ in $\left(-1\right)^m q^\frac{m\left(m-1\right)}{2}\prod_{j=0}^{m-1}\left(1-q^{m-j}\right)$ is the parity-count of the $m$-bounded distinct partitions of $h$.

\end{lemma}

\begin{proof}

This is clear from inspecting the way that each term is built when expanding the product.

\end{proof}

We now have a generalisation of the above lemma.

\begin{lemma}\label{distCoeff}

Let $h\le m$. Then the coefficient of $q^{m^2-h}$ in $\left(-1\right)^m q^\frac{m(m-1)}{2}\prod_{j=0}^{m-1-k}\left(1-q^{m-j}\right)$ is equal to the parity-count of $\mathcal{D}_2\left(h,k,\star\right)$.

\end{lemma}

\begin{proof}

Let $P_k(q)=\sum_{i=0}^{m^2}a^{(k)}_i q^i$, for some $k$ with $0\le k\le m$ have the following property (P):

 \begin{equation}\label{PC}
\tag{P}a^{(k)}_{m^2-h}=\mathrm{PC}\left(\mathcal{D}_2\left(h,k,\star\right)\right)\text{ for }h\le C+\frac{k(k+1)}{2}
\end{equation}

\noindent for some constant $C$. That is, $P_k(q)$ is a sort of partial generating function for these parity counts.

\noindent We have, from the remarks following Lemma \ref{dist2P}, that

\[\mathrm{PC}\left(\mathcal{D}_2\left(h,k,\star\right)\right)=\mathrm{PC}\left(\mathcal{D}_2\left(h-k,k,\star\right)\right)-\mathrm{PC}\left(\mathcal{D}_2\left(h-k,k-1,\star\right)\right)\]

\noindent and hence

\[\mathrm{PC}\left(\mathcal{D}_2\left(h-k,k-1,\star\right)\right)=\mathrm{PC}\left(\mathcal{D}_2\left(h-k,k,\star\right)\right)-\mathrm{PC}\left(\mathcal{D}_2\left(h,k,\star\right)\right)\]

\noindent Now suppose that $P_k(q)$ has property (\ref{PC}) and let $P_{k-1}(q)=P_k(q)\left(1-q^k\right)$. This implies $a^{(k-1)}_{m^2-(h-k)}=a^{(k)}_{m^2-(h-k)}-a^{(k)}_{m^2-h}$ for $h\le C+\frac{k(k+1)}{2}$, i.e. for $h-k\le C+\frac{(k-1)k}{2}$. By the above recurrence, therefore, $P_{k-1}(q)$ also has property (\ref{PC}) with the same constant $C$. Now let $P_m(q)=\left(-1\right)^m q^\frac{m(m-1)}{2}$. This has property (\ref{PC}) for at least $C=0$, as the number of $\left(h,m,\star\right)$-partitions is $0$ for $h<\frac{m(m+1)}{2}$ and $1$ for $h=\frac{m(m+1)}{2}$, and hence $\mathrm{PC}\left(\mathcal{D}_2\left(h,m,\star\right)\right)=0$ for  $h<\frac{m(m+1)}{2}$ and $\mathrm{PC}\left(\mathcal{D}_2\left(h,m,\star\right)\right)=\left(-1\right)^m$ for $h=\frac{m(m+1)}{2}$. Hence we know that property (\ref{PC}) holds with $C=0$ for all $P_k(q)$ down to $k=0$. However, we also know from Lemma \ref{distinctP} that for $k=0$, property (\ref{PC}) holds with $C=m$. We know that for $k=1$, property (\ref{PC}) holds with $C=0$; that is, we know the coefficients for $h=0,1$. Then we can use the recurrence relation $a^{(k-1)}_{m^2-(h-k)}=a^{(k)}_{m^2-(h-k)}-a^{(k)}_{m^2-h}$ with $k=1$ to extend this to lower-order coefficients until we run out of known coefficients in $P_0(q)$ at $h-k=m$, i.e. $h=m+1$. So property (\ref{PC}) holds with $C=m$ for $k=1$ also. We now repeat this with $P_2(q)$. We know the coefficients for $h\le3$, so we can use the recurrence to extend to lower orders (by jumps of $k$, i.e. $2$) until we run out of known coefficients in $P_1(q)$, at $h-k=m+1$, i.e. $h = m+1+2$.  So property (\ref{PC}) holds with $C=m$ for $k=2$ also. In this way we can continue extending until $k=m$, at each stage calculating enough coefficients to make property (\ref{PC}) hold with $C=m$. This is more than enough to establish the lemma.

\end{proof}

\begin{thm}\label{coeffs}

Let $C_{m,k}(q)$ be the (polynomial giving the) number of ways to extend a $k$-clique to a $k+1$-clique in $\mathbb{P}\left(M_m(q)\right)$. Let $k\le 3$ and $h\le m$. Then the coefficient of $q^{m^2-h}$ in $C_{m,k}(q)$ is equal to the coefficient of $q^h$ in $\prod_{i=1}^\infty\left(1-q^i\right)^{k-1}$.

\end{thm}

\begin{proof}

For $k=0$, $C_{m,k}(q)=\left\lbrack\begin{array}{c}2m\\m\end{array}\right\rbrack_q$. But the coefficient of $q^h$ in this is well-known to be the number of $m$-bounded partitions of $h$, and, as the $q$-binomial coefficient is symmetrical under $k\rightarrow m-k$, this is the same as the coefficient of $q^{m^2-h}$. Also, $\prod_{i=1}^\infty\left(1-q^i\right)^{-1}$ is well-known to be the generating function of the number of partitions, as is clear by expanding this expression, which coincides with the number of $m$-bounded partitions when $h\le m$.

For $k=1$, $C_{m,k}(q)=q^{m^2}$ and $\prod_{i=1}^\infty\left(1-q^i\right)^0=1$.

For $k=2$, $C_{m,k})q)=\left(-1\right)^m q^\frac{m\left(m-1\right)}{2}\prod_{j=0}^{m-1}\left(1-q^{m-j}\right)$, and as we have already established, the coefficient of $q^{m^2-h}$ in here is $\mathrm{PC}\left(\mathcal{D}_2\left(h,k,\star\right)\right)$ if $h\le m$. By essentially the same reasoning, this coincides with the coefficient of $q^h$ in $\prod_{i=1}^\infty\left(1-q^i\right)^1$.

For $k=3$, $C_{m,k}(q)=\left(-1\right)^mq^\frac{m(m-1)}{2}\displaystyle\sum_{i=0}^m\displaystyle\prod_{j=0}^{m-i-1}\left(1-q^{m-j}\right)$. By Lemma \ref{distCoeff}, this is the sum of $\mathrm{PC}\left(\mathcal{D}_2\left(h,k,\star\right)\right)$ for all $h\le m$, hence equal to the parity-count of all $2$-distinct partions of $h$ such that both subsets are $m$-bounded. For $h\le m$, the $m$-bounding is automatic, and it is not hard to see that $\prod_{i=1}^\infty\left(1-q^i\right)^2$ is the generating function for parity-counts of all $2$-distinct partitions.

\end{proof}

\begin{rem}The coefficients of $\prod_{i=1}^\infty\left(1-q^i\right)^{k-1}$ for $k=0,1,2,3$ are given by OEIS sequences A000041, A000007, A010815 and A002107 (\cite{OEIS}).

\end{rem}

\begin{corollary}

The leading term of $\cap k N$ for $\mathbb{P}\left(M_m(q)\right)$ is of degree $m^2-k$ for $1\le k\le 3$ and has coefficient $1$.

\end{corollary}

\begin{proof}

\[\cap k N=\sum_{i=0}^k\left(-1\right)^i\left(\begin{array}{c}k\\i\end{array}\right)C_{m,k}(q)\]

For sufficiently large $m$, the leading coefficients of $C_{m,k}(q)$ correspond to the first few coefficients of $\prod_{i=1}^\infty\left(1-q^i\right)^{k-1}$, so the leading coefficients of $\cap k N$ correspond to the first few coefficients of $\sum_{j=0}^k\left(-1\right)^j\left(\begin{array}{c}k\\j\end{array}\right)\prod_{i=1}^\infty\left(1-q^i\right)^{j-1}=\prod_{i=1}^\infty\left(1-q^i\right)^{-1}\left(1-\prod_{i=1}^\infty\left(1-q^i\right)\right)^k$. Since the lowest-degree non-zero term of $1-\prod_{i=1}^\infty\left(1-q^i\right)$ is $q$, the lowest degree of its powers successively increases by $1$, and hence the degree of the leading term of $C_{m,k}(q)$ decreases with $k$ in the same manner.

More explicitly, the coefficients of $C_{m,k}(q)$ start off like this:

\begin{center}
$\begin{array}{|c|c|c|c|c|c|}
\hline
&\mathbf{m^2}&\mathbf{m^2-1}&\mathbf{m^2-2}&\mathbf{m^2-3}&\mathbf{m^2-4}\\\hline
C_{m,0}&1&1&2&3&5\\\hline
C_{m,1}&1&0&0&0&0\\\hline
C_{m,2}&1&-1&-1&0&0\\\hline
C_{m,3}&1&-2&-1&2&1\\\hline
\end{array}$
\end{center}

Hence the coefficients of $\cap k N$ start off like this:

\begin{center}
$\begin{array}{|c|c|c|c|c|c|}
\hline
&\mathbf{m^2}&\mathbf{m^2-1}&\mathbf{m^2-2}&\mathbf{m^2-3}&\mathbf{m^2-4}\\\hline
\cap1N&0&1&2&3&5\\\hline
\cap2N&0&0&1&3&5\\\hline
\cap3N&0&0&0&1&4\\\hline
\end{array}$
\end{center}

It only remains to check the cases where $m<k$. We have already done this for $k=1$ and $k=2$ in the previous section. It is also easy to check for $k=3$. Indeed, $C_{0,3}=1$, $C_{1,3}=q-2$ and $C_{2,3}=q^4-2q^3-q^2+3q$, so where the coefficients are present, they have the needed values.

\end{proof}

The proof of Proposition \ref{coeffs} gives the impression that the theorem really consists of four unrelated facts, one for each value of $k$. It seems unlikely that this impression is accurate. Here is a sketch of how this pattern arises, showing how the $k=1$ case should imply the $k=2$ case, and the $k=2$ case should imply the $k=3$ case. In fact, the same mechanism should apply to any clique consisting of $\left(1,0\right)$ together with elements of the form $\left(\lambda I,1\right)$ for $\lambda\in\mathbb{F}_q$.

A function such that, for $h\le m$, its coefficient for the term of order $q^{m^2-h}$ is the same as the coefficient of $q^h$ in $\prod_{i=1}^\infty\left(1-q^i\right)^{k-1}$ is $q^{m^2}\prod_{i=1}^\infty\left(1-q^{-i}\right)^{k-1}$. Suppose $C_{m,k}$ is sufficiently approximated by this form. Now suppose we are moving up to $k+1$, and suppose we are choosing among matrices that are candidates for extending a $k-1$-clique to a $k$-clique, but are not necessarily candidates for extending to a $k+1$-clique---on account of capturing a non-trivial subspace. Then let us suppose that a matrix that captures the subspace where all coordinates after the $i$th are $0$ is of the form

\[\left(\begin{array}{c|c}A&*\\\hline B&G\end{array}\right)\]

\noindent where $A$ and $B$ are fixed, of sizes $i\times i$ and $\left(m-i\right)\times i$ respectively, $*$ can be anything, of size $i\times\left(m-i\right)$, and $G$ is one of the $C_{m-i,k}$ matrices of size $\left(m-i\right)\times\left(m-i\right)$ that would be candidates for a clique-extension in a space of dimension $m-i$. Now, by hypothesis, $C_{m-i,k}$ is approximated by $q^{\left(m-i\right)^2}\prod_{j=0}^\infty\left(1-q^{-j}\right)^{k-1}$, while the number of possible submatrices $*$ is $q^{i\left(m-i\right)}$, so from Lemma \ref{incexc}, the count of extensions is (approximately)

\[\begin{array}{ll}&\displaystyle\sum_{i=0}^m\left(-1\right)^i\left\lbrack\begin{array}{c}m\\k\end{array}\right\rbrack_q q^\frac{i(i-1)}{2}q^{i(m-i)}q^{\left(m-i\right)^2}\prod_{j=1}^\infty\left(1-q^{-j}\right)^{k-1}\\
=&\displaystyle\sum_{i=0}^m\left(-1\right)^i\left\lbrack\begin{array}{c}m\\k\end{array}\right\rbrack_q q^\frac{i(i-1)}{2}q^{m^2}q^{-m i}\prod_{j=1}^\infty\left(1-q^{-j}\right)^{k-1}\\
=&\displaystyle q^{m^2}\sum_{i=0}^m\left(-q^{-m}\right)^i\left\lbrack\begin{array}{c}m\\k\end{array}\right\rbrack_q q^\frac{i(i-1)}{2}\prod_{j=1}^\infty\left(1-q^{-j}\right)^{k-1}\\
=&\displaystyle q^{m^2}\prod_{i=0}^{m-1}\left(1-q^{-m}q^i\right)\prod_{j=1}^\infty\left(1-q^{-j}\right)^{k-1}\\
=&\displaystyle q^{m^2}\prod_{i=1}^{m}\left(1-q^{-i}\right)\prod_{j=1}^\infty\left(1-q^{-j}\right)^{k-1}
\end{array}
\]

\noindent (where the third step comes from the $q$-binomial theorem).

However,
\[\begin{array}{ll}
&\displaystyle q^{m^2}\prod_{i=1}^{m}\left(1-q^{-i}\right)\prod_{j=1}^\infty\left(1-q^{-j}\right)^{k-1}\\
\approx&\displaystyle q^{m^2}\prod_{i=1}^\infty\left(1-q^{-i}\right)\prod_{j=1}^\infty\left(1-q^{-j}\right)^{k-1}\\
=&\displaystyle \prod_{j=1}^\infty\left(1-q^{-j}\right)^k
\end{array}\]

\section*{Acknowledgements}

I would like to thank John Baez for many helpful comments that improved this paper in several ways.

\begin{appendices}

\section{A combinatorial identity} \label{lacun}

Here we prove the remark that we have $p\vert\cap n\mathrm{N}$ for every prime $p\le n$, where

\[\cap n\mathrm{N}=\left\vert J\right\vert\displaystyle\sum_{k=0}^n\left(-1\right)^k\left(\begin{array}{ccc}n\\k\end{array}\right)\displaystyle\prod_i\left(q_i+1-k\right)\]

Observe that the sum is of the form $\sum_k\left(-1\right)^k\left(\begin{array}{ccc}n\\k\end{array}\right)f(k)$, where $f(k)$ is a function whose value mod $p$ depends only on $k$ mod $p$ for every $p$, so that the sum separates into terms of the form $\sum_j\left(-1\right)^{j p+m}\left(\begin{array}{ccc}n\\j p+m\end{array}\right)f(m)$ (where $j$ ranges over all values such that $0\le j p+m\le n$). But we have the following theorem.

\begin{thm}
$\displaystyle\sum_{\substack{j\\0\le j p+m\le n}}\left(-1\right)^{j p+m}\left(\begin{array}{ccc}n\\j p+m\end{array}\right)=0$ mod $p$ for prime $p\le n$
\end{thm}

\begin{proof}
Let $\phi$ be a primitive $p$th root of unity.

\begin{center}
$\begin{array}{lll}
&\displaystyle\frac{1}{p}\displaystyle\sum_{j=0}^{p-1}{\left(1-\phi^j\right)^n\phi^{-j m}}\\
=&\displaystyle\frac{1}{p}\displaystyle\sum_{j=0}^{p-1}{\displaystyle\sum_{k=0}^n{\left(-1\right)^k\left(\begin{array}{c}n\\k\end{array}\right)\phi^{j k}\phi^{-j m}}}\\
=&\displaystyle\sum_{k=0}^n{\left(-1\right)^k\left(\begin{array}{c}n\\k\end{array}\right)\left\lbrace\displaystyle\frac{1}{p}\displaystyle\sum_{j=0}^{p-1}{\phi^{j\left(k-m\right)}}\right\rbrace}
\end{array}$
\end{center}

The term in braces is $0$ unless $\left(k-m\right)$ is a multiple of $p$, in which case it is $1$, so the whole expression is equal to the alternating lacunary sum above.

But on the other hand, if $p\le n$, then every term in the first sum contains a factor of the form

\[\left(1-\phi^j\right)^p\]

But, mod $p$, this is equal to $1-\phi^{j p}=0$.\end{proof}

\section{Inequivalent cliques} \label{extensions}

Here is an example of why we can not expect to have a general formula for the number of ways to extend a $k$-clique for $k>3$. Suppose $m=2$ and $q=3$. Let us work in $\mathrm{GL}_2(3)$, where (to abuse the definition a little) we will say that two elements are distant if their difference is invertible. Suppose we have already picked the following $3$-clique:

\[\left\{\left(\begin{array}{cc}1&0\\0&1\end{array}\right),\left(\begin{array}{cc}2&0\\0&2\end{array}\right),\left(\begin{array}{cc}0&2\\1&0\end{array}\right)\right\}\]

There are $9$ elements which will extend this to a $4$-clique, and they fall into $3$ classes:

\[\mathbf{A}=\left\{\left(\begin{array}{cc}2&2\\2&1\end{array}\right),\left(\begin{array}{cc}2&1\\1&1\end{array}\right),\left(\begin{array}{cc}1&2\\2&2\end{array}\right),\left(\begin{array}{cc}1&1\\1&2\end{array}\right)\right\}\]

\[\mathbf{B}=\left\{\left(\begin{array}{cc}2&2\\1&2\end{array}\right),\left(\begin{array}{cc}2&1\\2&2\end{array}\right),\left(\begin{array}{cc}1&2\\1&1\end{array}\right),\left(\begin{array}{cc}1&1\\2&1\end{array}\right)\right\}\]

\[\mathbf{C}=\left\{\left(\begin{array}{cc}0&1\\2&0\end{array}\right)\right\}\]

All the elements of $\mathbf{A}$ are mutually distant, all the elements of $\mathbf{B}$ are mutually distant, and the single element of $\mathbf{C}$ is distant to everything in $\mathbf{A}$ and $\mathbf{B}$; but no element of $\mathbf{A}$ is distant to any element of $\mathbf{B}$. This means that at this point, we are committed to forming one of two maximal cliques (with $8$ elements), {\it viz.} by appending either $\mathbf{A}\cup\mathbf{C}$ or $\mathbf{B}\cup\mathbf{C}$ to our existing $3$-clique. Now, if we form a $4$-clique by appending the single element of $\mathbf{C}$, the resulting clique can be extended to a $5$-clique in $8$ ways (by appending any element of $\mathbf{A}$ or $\mathbf{B}$). However, if instead we choose an element of $\mathbf{A}$, the resulting $4$-clique will only be extensible to a $5$-clique in $4$ ways (by appending either the element of $\mathbf{C}$ or another element of $\mathbf{A}$), and similarly for $\mathbf{B}$. So there are (at least) two kinds of $4$-clique, with different extension counts.

If we extend by $\mathbf{A}\cup\mathbf{C}$, then we get a maximal clique not generated as powers of a single matrix.

\section{Inextensible submaximal clique} \label{inextensible}

Here is a clique in $\mathrm{GL}_2\left(\mathbb{F}_5\right)$ that is inextensible but has only $20$ elements.

\[\begin{array}{cccc}
\left(\begin{array}{cc}0&1\\1&2\end{array}\right)&\left(\begin{array}{cc}0&2\\2&1\end{array}\right)&\left(\begin{array}{cc}0&3\\3&1\end{array}\right)&\left(\begin{array}{cc}0&4\\4&2\end{array}\right)\\\left(\begin{array}{cc}1&0\\0&1\end{array}\right)&\left(\begin{array}{cc}1&1\\1&3\end{array}\right)&\left(\begin{array}{cc}1&2\\2&0\end{array}\right)&\left(\begin{array}{cc}1&3\\3&0\end{array}\right)\\\left(\begin{array}{cc}1&4\\4&3\end{array}\right)&\left(\begin{array}{cc}2&0\\0&2\end{array}\right)&\left(\begin{array}{cc}2&1\\1&0\end{array}\right)&\left(\begin{array}{cc}2&2\\2&3\end{array}\right)\\\left(\begin{array}{cc}2&3\\3&3\end{array}\right)&\left(\begin{array}{cc}2&4\\4&0\end{array}\right)&\left(\begin{array}{cc}3&0\\0&3\end{array}\right)&\left(\begin{array}{cc}3&1\\1&1\end{array}\right)\\\left(\begin{array}{cc}3&2\\2&2\end{array}\right)&\left(\begin{array}{cc}3&3\\3&2\end{array}\right)&\left(\begin{array}{cc}3&4\\4&1\end{array}\right)&\left(\begin{array}{cc}4&0\\0&4\end{array}\right)
\end{array}\]

\section{The limit $q\rightarrow1$}

Theorems and formulas that hold generally over finite fields $\mathbb{F}_q$ also often have a true combinatorial interpretation in the case $q=1$. This is the case with our counting formulas.

We can define (e.g. Section 5.1 of \cite{F1}) the general linear group $\mathrm{GL}_n\left(\mathbb{F}_1\right)$ over the fictitious one-element field to consist of $n\times n$ permutation matrices---those with a single entry of $1$ in each row and each column, all other entries being $0$. These matrices act by permuting rows (from the left) or columns (from the right) and  $\mathrm{GL}_n\left(\mathbb{F}_1\right)$ is isomorphic to the symmetric group $S_n$. We can go on to define $\mathrm{GL}_n\left(\mathbb{F}_{1^n}\right)$ over ``extensions of the one-element field", which are the same as elements of $\mathrm{GL}_n\left(\mathbb{F}_1\right)$ except that instead of entries being $1$, they may be any $n$th root of unity; this group is isomorphic to the wreath product of $S_n$ and the cyclic group $Z_n$.

A point in $\mathbb{P}\left(\mathrm{M}_m\left(\mathbb{F}_1\right)\right)$ consists of the first $m$ rows of a $2m\times2m$ permutation matrix, modulo multiplication from the left by $m\times m$ permutation matrices. Since the latter permute rows arbitrarily, we only care which columns contain a $1$, not which row they appear in, so the number of points is $\left(\begin{array}{c}2m\\m\end{array}\right)$, i.e. the $q\rightarrow1$ limit of $\left\lbrack\begin{array}{c}2m\\m\end{array}\right\rbrack_q$. Given a choice for the filled columns of the first $m$ rows, there is only one choice for the remaining $m$ rows, {\it viz.} the remaining unfilled places, so only one point is distant to a given point. This is the $q\rightarrow1$ limit of $q^{m^2}$, as we would hope. There are then of course no extensions to cliques of $3$ or $4$ points, and this is again given by the $q\rightarrow1$ limits of the clique extension polynomials, which equal $0$. Moving to extensions of $\mathbb{F}_1$ results in no change: multiplication from the left by elements of $\mathrm{GL}_m\left(\mathbb{F}_{1^n}\right)$ allows us not only to arbitrarily permute rows but also to reduce all non-zero entries to $1$.

Hence the distant graph of $\mathrm{M}_m\left(\mathbb{F}_{1^n}\right)$ consists of $\left(\begin{array}{c}2m\\m\end{array}\right)$ points, arranged into mutually distant pairs.

\end{appendices}
\end{document}